\documentclass[12pt]{amsart}

\usepackage{breakurl}
\usepackage{graphicx}
\usepackage{amsmath}
\usepackage{amscd}
\usepackage{amsfonts}
\usepackage{amssymb}
\usepackage{color}
\usepackage{fullpage}
\usepackage{mathtools}
\usepackage[pagebackref,hypertexnames=false, colorlinks, citecolor=red,linkcolor=blue, urlcolor=red]{hyperref}
\usepackage{todonotes}
\usepackage[inline]{enumitem}
\usepackage{cleveref}

\numberwithin{equation}{section}

\newtheorem{theorem}{Theorem}[section]
\newtheorem{lemma}[theorem]{Lemma}

\theoremstyle{definition}

\newtheorem{remark}[theorem]{Remark}
\newtheorem{definition}[theorem]{Definition}

\newcommand{\be}{\begin{equation}}
\newcommand{\ee}{\end{equation}}
\newcommand{\bes}{\begin{equation*}}
\newcommand{\ees}{\end{equation*}}

\newcommand{\cB}{\mathcal{B}}

\newcommand{\cF}{\mathcal{F}}
\newcommand{\cG}{\mathcal{G}}
\newcommand{\cH}{\mathcal{H}}

\newcommand{\cK}{\mathcal{K}}
\newcommand{\cL}{\mathcal{L}}

\newcommand{\cU}{\mathcal{U}}

\newcommand{\bN}{\mathbb{N}}

\newcommand{\bR}{\mathbb{R}}

\newcommand{\im}{\operatorname{Im}}

\newcommand{\dom}{\operatorname{dom}}

\begin{document}

\title[Bounded perturbations via dilations]{Bounded perturbations of the Heisenberg commutation relation via dilation theory}

 \author{Malte Gerhold}
 \address{M.G.\\ Department of Mathematical Sciences\\ NTNU Trondheim\\ Norway}
 \email{malte.gerhold@ntnu.no}
\urladdr{\href{https://sites.google.com/view/malte-gerhold/}{\url{https://sites.google.com/view/malte-gerhold/}}}

 \author{Orr Moshe Shalit}
 \address{O.S., Faculty of Mathematics\\
 Technion - Israel Institute of Technology\\
 Haifa\; 3200003\\
 Israel}
 \email{oshalit@technion.ac.il}
 \urladdr{\href{https://oshalit.net.technion.ac.il/}{\url{https://oshalit.net.technion.ac.il/}}}

 \thanks{The work of M. Gerhold was carried out during the tenure of an ERCIM ‘Alain
Bensoussan’ Fellowship Programme.}
 \thanks{The work of O.M. Shalit is partially supported by ISF Grant no.\ 431/20.
 }


\begin{abstract}
We extend the notion of {\em dilation distance} to strongly continuous one-parameter unitary groups.
If the dilation distance between two such groups is finite, then these groups can be represented on the same space in such a way that their generators have the same domain and are in fact a bounded perturbation of one another. 
This result extends to $d$-tuples of one-parameter unitary groups. 
We apply our results to the Weyl canonical commutation relations, and as a special case we recover the result of Haagerup and R{\o}rdam that the infinite ampliation of the canonical position and momentum operators satisfying the Heisenberg commutation relation are a bounded perturbation of a pair of strongly commuting selfadjoint operators. We also recover Gao's higher-dimensional generalization of Haagerup and  R{\o}rdam's result, and in typical cases we significantly improve control of the bound when the dimension grows.
\end{abstract}

\maketitle

%

\section{Dilation of groups and perturbations of generators}

Haagerup and R{\o}rdam constructed in \cite{HR95} two norm continuous families $\theta \mapsto u_\theta, \theta \mapsto v_\theta$ of unitaries on a Hilbert space $\cH$ such that $v_\theta u_\theta = e^{i \theta} u_\theta v_\theta$ for all $\theta$. 
This is a surprising result, in light of the fact that the C*-algebra of the CCR is never separable \cite[Proposition 2.2]{PetzBook}. 
One key step in their intricate proof was a perturbation theorem for unbounded operators, namely that the infinite ampliation of the canonical position and momentum operators satisfying the Heisenberg commutation relation are a bounded perturbation of a pair of strongly commuting selfadjoint operators. 
In \cite{GPSS21}, among other things, we used the notion of {\em dilation distance} in order to give a new proof of the existence of the continuous families $\theta \mapsto u_\theta, \theta \mapsto v_\theta$ that avoids the use of unbounded operators. 
In this note we adapt our dilation techniques to recover Haagerup and R{\o}rdam's result on perturbations of unbounded operators. 
Along the way we record some interesting connections between dilation theory and strongly continuous unitary groups. 

The basic objects of study in this article are $d$-tuples of (strongly continuous, one-para\-meter) unitary groups. 
Recall that a unitary group is a group homomorphisms $u \colon \bR \to \cU(\cH)$ from $\bR$ into the unitary group $\cU(\cH)$ for some separable Hilbert space $\cH$ such that the map $t \mapsto u(t)$ is continuous in the strong operator topology. 
We shall freely use well known facts about groups and their generators,  and in particular that a unitary group $(u(t))_{t\in \bR}$ gives rise to a densely defined (typically unbounded) selfadjoint operator $A$ on $\cH$ determined by $iAx = \lim_{t \to 0} \frac{u(t)x-x}{t}$ for all $x \in \dom A$, that every such selfadjoint operator generates a unitary group by the functional calculus as $u(t) = \exp(itA)$, and that these two associations are mutual inverses (this is {\em Stone's theorem}; see e.g.\ \cite[Theorem 10.15]{HallBook}). 


\begin{definition}
Two $d$-tuples of one-parameter unitary groups $u = (u_1, \ldots, u_d)$ on $\cH$ and $v = (v_1, \ldots, v_d)$ on $\cK$ are said to be {\em equivalent} if there are faithful normal representations $\pi \colon \cB(\cH) \to \cB(\cL)$ and $\rho \colon \cB(\cK) \to \cB(\cL)$ such that $\pi(u_k(t)) = \rho(v_k(t))$ for all $k=1,\ldots,d$ and all $t\in \bR$. 
In this case we write $u \sim v$.
\end{definition}
Recall that every faithful normal representation $\pi$ of $\cB(\cH)$ on $\cG$ is of the form
\[
\pi(a)=W^*(a\otimes 1)W
\]
for some Hilbert space $\cF$ and unitary $W\colon \cG\to \cH\otimes \cF$. 
Indeed, every faithful normal representation of $\cB(\cH)$ is determined by a representation of the compacts $\cK(\cH)$, and every representation of $\cK(\cH)$ is unitarily equivalent to a multiple of the identity representation (see, for example, Corollary 1 to Theorem 1.4.4 in \cite{ArvBook}). 
From this it follows easily that unitary groups $u, v$ are equivalent if and only if their infinite ampliations $u^\infty:=u\otimes 1_{\ell^2(\mathbb N)},v^\infty:= v\otimes 1_{\ell^2(\mathbb N)}$ are unitarily equivalent, $u^\infty\sim_{\mathrm{ue}}v^{\infty}$.

The following definition is an adaptation, to tuples of unitary groups, of the {\em dilation distance} defined in \cite{GPSS21} for tuples of unitaries. 
Note carefully that here we use a stronger notion of equivalence, hence a more stringent definition of dilation. 

\begin{definition}
Given two $d$-tuples of one-parameter unitary groups $u = (u_1, \ldots, u_d)$ and $v = (v_1, \ldots, v_d)$, and given a positive real number $c$, we write $u \prec_c v$ if there exist two Hilbert spaces $\cH \subseteq \cK$ and two $d$-tuples of one-parameter unitary groups $U = (U_1, \ldots, U_d)$ on $B(\cH)^d$ and $V = (V_1, \ldots, V_d)$ on $B(\cK)^d$, such that $u \sim U$, $v \sim V$ and 
\[
U_k(t) = P_\cH C(t)V_k(t) \big|_\cH
\]
for all $k=1, \ldots, d$ and all $t$ in a neighborhood of $0$, where $C(t)$ is a family of constants such that $C(t) \leq e^{t^2c}$. 
The \emph{dilation distance} is defined as
\[
d_{\mathrm D}(u,v):=\max \left( \inf\{c: u\prec_c v \}, \inf\{c: v\prec_c u\} \right).
\]
\end{definition}
As usual, the above definition is understood to mean that $d_{\mathrm D}(u,v) = \infty$ in the case that there are no constants $c$ such that the semigroups dilate one another as above. 

Our first goal is to show that tuples that are close in dilation distance are also close in the sense that they have equivalent representations on the same space that are close in norm. 
The key is the technical \Cref{thm:DtoHR} below, which is proved roughly along the lines of \cite[Theorem 2.6]{GPSS21}, with some changes so as to make the method applicable to strongly continuous unitary groups (its proof can be adapted to the case of tuples of single unitaries to recover \cite[Theorem 2.6]{GPSS21} with improved constants).

The following simple preparatory lemma has no counterpart for tuples of unitaries. 

\begin{lemma}\label{lem:forget_second_order_term}
  Let $U$ and $V$ be one-parameter unitary groups on the same Hilbert space $\mathcal L$, and let $C,D\in \mathbb R_+$. Then the following are equivalent.
  \begin{enumerate}
  \item $\| U(t)-V(t)\| \leq C|t|+Dt^2$ for all positive $t$ in some neighborhood of $0$,
  \item $\| U(t)-V(t)\| \leq C|t|$ for all $t\in\mathbb R$.
  \end{enumerate}
\end{lemma}
\begin{proof}
  Assume $\| U(t)-V(t)\| \leq C|t|+Dt^2$
  for all positive $t$ in some neighborhood of $0$. Then
  for every $C'>C$, we have $\| U(t)-V(t)\| \leq C'|t|$ for all $t$ in some small interval $[0,a]$ with $a>0$. This estimate easily extends to all $t\in\mathbb R_+$; indeed, for $s,t\in[0,a]$ we have
  \[\| U(t+s)-V(t+s)\| = \|U(t)(U(s)-V(s)) + (U(t)-V(t))V(s) \|\leq C'(t+s),\]
  so the estimate holds on $[0,2a]$ and by induction it holds on $\mathbb R_+$. From $\| U(-t)-V(-t)\|= \bigl\|U(-t)\bigl(V(t)-U(t)\bigr)V(-t)\bigr\|=\|U(t)-V(t)\|$ we conclude that the estimate holds also for negative $t$ and hence on all of $\mathbb R$. Now, if $\|U(t)-V(t)\|\leq C'|t|$ for all $C'>C$ , then clearly $\|U(t)-V(t)\|\leq C|t|$. The reverse implication is trivial.
\end{proof}

\begin{theorem}\label{thm:DtoHR}
Let $u = (u_1, \ldots, u_d)$ and $v = (v_1, \ldots, v_d)$ be two $d$-tuples of one-parameter unitary groups on Hilbert spaces $\cH_1$ and $\cH_2$. 
If $d_{\mathrm D}(u,v)\leq\delta<\infty$, then there exist a Hilbert space $\cL$ and faithful normal representations $\pi\colon \cB(\cH_1) \to \cB(\mathcal L)$ and $\rho\colon \cB(\cH_2)\to \cB(\cL)$ of infinite multiplicity such that 
\[
\| \pi(u_k(t))-\rho(v_k(t))\| \leq 5\sqrt2|t|\delta^{1/2}
\]
for all $k=1, \ldots, d$ and all $t$ in a neighborhood of $0$. 
\end{theorem}

\begin{proof}
  
Assume first that $\delta > d_{\mathrm D}(u,v)$. 

By replacing $u$ and $v$ with their infinite ampliations, we may assume that there are unitaries 
\[
W_1\colon \cH_2\oplus \cK_2\to \cH_1,\quad W_2\colon \cH_1\oplus \cK_1\to\cH_2
\]
such that
\be
W^*_1u_k(t)W_1 =
\begin{pmatrix}
v_k'(t) & x_k(t) \\ y_k(t) & z_k(t) 
\end{pmatrix} \quad \textrm{and} \quad
W_2^*v_k(t)W_2 = 
\begin{pmatrix}
u'_k(t) & r_k(t) \\ s_k(t) & w_k(t) 
\end{pmatrix}
\ee
with
$v'(t) = C(t)^{-1} v(t)$ and $u' = D(t)^{-1} u(t)$, where $C(t) \leq e^{ct^2}$ and $D(t) \leq e^{dt^2}$  for some $c,d < \delta$. It follows that $\|v_k(t)-v'_k(t)\| \leq |1 - e^{-ct^2}| \leq ct^2 \leq \delta t^2$ for all $k=1,\ldots,d$ and all $t\in\bR$; likewise for $u$, we have $\|u_k(t)-u'_k(t)\| \leq \delta t^2$. 
By a further ampliation if needed, we may assume that every entry $x_k,y_k, z_k\ldots$ in the above operator matrices has infinite multiplicity. 

Since the $u_k$ and the $v_k$ are unitary groups, we find that
\[
y_k^* y_k = 1-v'^*_kv'_k = v^*_k(v_k - v'_k) + (v_k - v'_k)^*v'_k
\]
and, since $\|v'_k(t)\|\leq 1$ and $\|v_k(t)\|=1$, we conclude
$\|y_k(t)\|\leq |t|\sqrt{2\delta}$.
Analogous reasoning yields $\|x_k(t)\| , \|s_k(t)\| , \|r_k(t)\| \leq |t|\sqrt{2\delta}$. With the block matrices
\[E:=
\begin{pmatrix}
  v'-v &x\\
  y & 0
\end{pmatrix},\quad
F:=
\begin{pmatrix}
  u'-u &r\\
  s&0
\end{pmatrix},\]
we obtain
\begin{align}
\label{eq:equiv}
  W_1^*uW_1= v\oplus z + E,\quad W_1^*vW_2= u\oplus w +F.
\end{align}
The norm of $E$ can be estimated as
\[\|E_k(t)\|\leq \max (\|x_k(t)\|,\|y_k(t)\|) +  \|(v'_k(t)-v_k(t))\| \leq  |t|\sqrt{2\delta} + t^2\delta,
\]
and, analogously, $\|F_k(t)\|\leq |t|\sqrt{2\delta}+t^2\delta$.
Define
\begin{align*}
  \pi^{(1)} &\colon \cB(\cH_1) \to \cB(\cH_2 \oplus \cK_2),\quad \pi^{(1)}(a):=W_1^*aW_1,\\
 \rho^{(1)} &\colon  \cB(\cH_2) \to \cB(\cH_1 \oplus \cK_1),\quad \rho^{(1)}(b):=W_2^*bW_2
\end{align*}
so that 
\[
\pi^{(1)}(u) =v\oplus z + E \quad \textrm{and} \quad
\rho^{(1)}(v) = u\oplus w + F.
\]
Let $\pi^{(k)}$ and $\rho^{(k)}$ be conjugation with $W_1\oplus\mathrm{id}$ and $W_2\oplus\mathrm{id}$, respectively, where $\mathrm{id}$ denote identity operators on direct sums of $k-1$ Hilbert spaces which will be clear from the context. Then we obtain 
\begin{align*}
\pi^{(2)}\rho^{(1)}(v) = \pi^{(1)}(u)\oplus w + \pi^{(2)}(F) = v\oplus z\oplus w + E\oplus 0_{\cK_1} + \pi^{(2)}(F) .
\end{align*}
On the other hand, we find that 
\begin{align*}
  \pi^{(3)} \rho^{(2)}\pi^{(1)}(u)
  &= \pi^{(3)} \rho^{(2)} ( v \oplus z ) + \pi^{(3)} \rho^{(2)} (E) \\
  &= \pi^{(3)} (u \oplus w \oplus z) + \pi^{(3)}(F\oplus 0_{\cK_2}) +  \pi^{(3)} \rho^{(2)} (E)\\
  &= v \oplus z \oplus w \oplus z + E\oplus 0_{\cK_1\oplus\cK_2} + \pi^{(2)}(F)\oplus 0_{\cK_2} +  \pi^{(3)} \rho^{(2)} (E) .
\end{align*}
To summarize more concisely what we found, let $\sim_{\mathrm {ue}}$ denote unitary equivalence and write 
\[
v \sim_{\mathrm {ue}} (v\oplus z\oplus w) + R 
\]
and
\[
u \sim_{\mathrm {ue}} (v\oplus z\oplus w \oplus z) + S 
\]
where $\|R_k(t)\| + \|S_k(t)\|\leq 3\|E_k(t)\|+2\|F_k(t)\|\leq 5(|t|\sqrt{2\delta}+t^2\delta)$.

Recall that in the beginning of the proof we passed to infinite ampliations, therefore $z\oplus z \sim_{\mathrm {ue}}  z$ and we obtain faithful normal representations
 \[
\pi\colon \cB(\cH_1) \to \cB(\mathcal L), \rho\colon \cB(\cH_2)\to \cB(\cL) ,\quad  \cL:=\cH_2\oplus\cK_2\oplus\cK_1 
\]
with
\[  
\|\pi(u_k(t))-\rho(v_k(t))\|\leq 5(|t|\sqrt{2\delta}+t^2\delta)
\]
in a neighborhood of $0$.

From \Cref{lem:forget_second_order_term}, we conclude that $\|\pi(u_k(t))-\rho(v_k(t))\|\leq 5\sqrt{2}|t|\delta^{1/2}$ for all $t$ and all $\delta > d_{\mathrm D}(u,v)$, hence also for $\delta= d_{\mathrm D}(u,v)$.
%
%
\end{proof}

Our next task is to relate closeness of groups to closeness of their generators. The next two lemmas might be well-known, but we include our own proofs of them for completeness.\footnote{Note that \Cref{lem:bdd_prt} is the sufficiency statement in \cite[Proposition 2.3]{Gao18}, but our proof is somewhat different, presenting it as an immediate consequence of the characterization of the domain of the generator of a one-parameter unitary group given in \Cref{lem:dom}.}

\begin{lemma}\label{lem:dom}
Let $u$ be a one-parameter unitary group with selfadjoint generator $A$. 
Then
\[
\dom A = \left\{x:\limsup_{t\to 0} \frac{\|(u(t)-1)x\|}{t}<\infty \right\}.
\]
\end{lemma}

\begin{proof}
Let $x$ be such that $\limsup \frac{\|(u(t)-1)x\|}{t}<\infty$. 
For all $y$ from the dense subspace $\dom(A)$ we have
\begin{align*}
|\langle x,A y\rangle|
=\lim |\langle x, {\textstyle\frac{1}{t}}(u(t)-1)y\rangle|\leq\limsup |\langle \frac{1}{t}(u(t)^* - 1)x,y \rangle|\leq \limsup \frac{\|(u(t)-1)x\|}{t} \|y\|,    
\end{align*}
therefore $y\mapsto |\langle x,Ay\rangle|$ is bounded and we conclude that $x\in\dom A^* =  \dom A$. The reverse inclusion is obvious.
\end{proof}

\begin{lemma}\label{lem:bdd_prt}
Let $u$ and $v$ be one-parameter unitary groups on a Hilbert space $\cH$ with selfadjoint generators $A,B$. 
Assume further that there is a constant $K$ such that $\|u(t)-v(t)\|\leq K|t|$ for all $t$ in neighborhood of $0$. 
Then $\dom A=\dom B$ and $\|A-B\|\leq K$, that is, there exists a bounded operator $C \in \cB(\cH)$ such that $B = A + C$ and $\|C\|\leq K$. 
\end{lemma}

\begin{proof}
Let $x \in \cH$. 
By \Cref{lem:dom}, 
\[
\limsup \frac{\|(v(t)-1)x\|}{t}\leq \limsup \frac{\|(u(t)-1)x\|}{t}+ K\|x\|
\]
shows that $\dom B\subset\dom A$ and the reverse inclusion follows analogously. Since the domains agree, $C:=A-B$ is a densely defined operator. For $x\in\dom A$, we have
\[
\|Cx\|\leq\limsup \frac{\|(u(t)-v(t))x\|}{t} \leq K\|x\|.
\]
But since $\dom A$ is dense in $\cH$, the operator $C$ extends to a bounded operator on $\cH$ of norm at most $K$. 
\end{proof}

We now combine the results of this section to show that unitary groups that are at a finite dilation distance from each other are, in fact, bounded perturbations of one another. 
\begin{theorem}\label{thm:bdd_pert}
Let $u$ and $v$ be two $d$-tuples of one-parameter unitary groups. 
If $d_{\mathrm D}(u,v)\leq \delta <\infty$, then there are realizations as $u\sim U$ and $v\sim V$ of infinite multiplicity on the same Hilbert space $\cH$, and there are selfadjoint operators $A_1, \ldots, A_d$ and $B_1, \ldots, B_d$ in $\cH$  with $\dom A_k = \dom B_k$ such that $U_k(t)=\exp(itA_k)$ and $V_k(t) = \exp(itB_k)$ and such that 
\[
\|B_k - A_k\| \leq 5\sqrt{2}\, \delta^{1/2}
\]
for all $k=1,\ldots, d$. 
 \end{theorem}
\begin{proof}
By \Cref{thm:DtoHR}, there are realizations $u\sim U$ and $v \sim V$ such that $\|V_k(t) - U_k(t)\|\leq K|t|$ for all $k=1, \ldots, d$, where $K = 5\sqrt{2}\, \delta^{1/2}$.
Applying \Cref{lem:bdd_prt} for all $k=1,\ldots, d$ we obtain the result. 
\end{proof}

\begin{remark}
The converse fails in an obvious manner: in $\mathbb C$, the unitary groups $e^{iat},e^{ibt}$ with distinct generators $a,b\in\mathbb R$ have infinite dilation distance. The notion of dilation distance works poorly when the matrix ranges do not contain open subsets. 
\end{remark} 

\section{Bounded perturbation of the Heisenberg commutation relations}

A pair of selfadjoint (unbounded) operators $P$ and $Q$ on a Hilbert space is said to satisfy the {\em Heisenberg commutation relation} (or sometimes the {\em Weyl form} of the canonical commutation relations) if the corresponding unitary groups $u(t) = e^{itP}$ and $v(t) = e^{itQ}$ satisfy
\be\label{eq:q-commuting}
u(s)v(t) = e^{ist} v(t)u(s),
\ee
a condition which is customarily interpreted as 
\[
[P,Q] = -i I. 
\]
It is known \cite[Theorem 14.8]{HallBook} that that there exists a unique (up to unitary equivalence) representation of the Heisenberg commutation relation that is irreducible, in the sense that the commutant of $C^*(u(s),v(t) : s,t \in \bR)$ is trivial, namely: the essentially unique representation on $L^2(\bR)$ in which $P = -i \frac{\mathrm d}{\mathrm dx}$ is the momentum operator and $Q = M_x$ (where $(M_x f)(x) = xf(x)$) is the position operator.
Haagerup and R\o rdam proved that the infinite ampliations of $P$ and $Q$ are a bounded perturbation of a pair of selfadjoint operators $P_0$ and $Q_0$ that strongly commute, in the sense that the one-parameter groups that they generate satisfy $e^{isP_0} e^{itQ_0} =  e^{isQ_0} e^{itP_0}$ for all $s,t \in \bR$ \cite[Theorem 3.1]{HR95}. 
A higher dimensional generalization of this was obtained by Gao \cite{Gao18}. 
In this section, using our dilation framework, we recover this result of Gao.


\begin{definition}
Given a real antisymmetric $d\times d$ matrix $\Theta=(\theta_{k,\ell})$, we say that a tuple $u = (u_1, \ldots, u_d)$ of one-parameter unitary groups {\em commutes according to $\Theta$} if
\be\label{eq:Theta_commuting}
u_\ell(s) u_k(t)=e^{i\theta_{k,\ell}st}u_k(t) u_\ell(s) \,\, , \,\, k,\ell = 1, \ldots, d \,\, , \,\, s,t \in \bR. 
\ee
\end{definition}

A concrete representation of $\Theta$-commuting one-parameter unitary groups is given by the Weyl unitaries (see, e.g., \cite{Par12}), which have been observed in \cite{GS20} to behave nicely under compressions. 
For a Hilbert space $H$ let
\[
\Gamma(H):=\bigoplus_{k=0}^{\infty} H^{\otimes_s k}
\]
be {\em the symmetric Fock space} over $H$.
The {\em exponential vectors} $e(x):=\sum_{k=0}^{\infty}\frac{1}{\sqrt{k!}} x^{\otimes k}, x\in H$ form a linearly independent and total subset of $\Gamma(H)$.
Clearly, $\langle e(x),e(y)\rangle = e^{\langle x,y\rangle}$ for all $x,y\in H$. 
For $z\in H$ the {\em Weyl operator} $W(z)\in \cB(\Gamma(H))$ is defined by
\[
W(z) e(x)=e(z+x) \exp\left(-\frac{\|z\|^2}{2} - \langle x,z\rangle \right)
\]
for all exponential vectors $e(x)$.
Simple calculations show that the Weyl operators are unitary and that $W(z),W(y)$ commute up to the phase factor $e^{2i \im \langle y,z\rangle}$, that is
\[
W(y) W(z) = e^{2i \im \langle y,z\rangle} W(z) W(y).
\]


\begin{lemma}\label{lem:exitThetacom}
Let $\Theta=(\theta_{k,\ell})$ be a real and antisymmetric $d\times d$ matrix. 
\begin{enumerate}
\item There exists a Hilbert space $H$, $\dim H=d$ and linearly independent $x_1,\ldots, x_d\in H$ such that
\[
2\im\langle x_\ell,x_k\rangle = \theta_{k,\ell}.
\]
\item The families $u_k(t) = W(t x_k)$ for $k=1, \ldots, d$ and $t \in \bR$ give rise to a $d$-tuple of strongly continuous one-parameter unitary groups that satisfiy \eqref{eq:Theta_commuting}. 
\end{enumerate}
\end{lemma}
\begin{proof}
The first part follows directly from \cite[Lemma 5.2]{GPSS21}. 
A routine argument shows that $u_1, \ldots, u_d$ are strongly measurable hence strongly continuous unitary groups and that
\[
u_\ell(s) u_k(t) = W(s x_\ell) W(t x_k) = e^{i  2 \im \langle s x_\ell, tx_k \rangle} W(t x_k) W(t x_\ell) = e^{i\theta_{k,\ell}st}u_k(t) u_\ell(s),
\]
i.e., \eqref{eq:Theta_commuting} is satisfied. 
\end{proof}

\begin{lemma}\label{lem:vectors-xyz}
Let $\Theta=(\theta_{k,\ell})$, $\Theta'=(\theta'_{k,\ell})$ be real and antisymmetric $d\times d$ matrices, and let $x_1, \ldots, x_d$ be as in \Cref{lem:exitThetacom}. 
Then there exist a Hilbert space $K\supset H$, $\dim K = 2d$ and linearly independent vectors $z_1,\ldots, z_d\in K$, such that $x_k:=p_H z_k$ ($k=1, \ldots, d$, $p_H$ the projection onto $H$) and such that with $y_k:=p_{H^{\perp}}z_k$ the following conditions are satisfied:
\begin{center}
  \begin{enumerate*}[label=\textnormal{
    }]
  \item $2\im\langle z_\ell,z_k\rangle = \theta_{k,\ell}'$;
  \item $2\im\langle x_\ell,x_k\rangle = \theta_{k,\ell}$;
  \item $\|y_k\|^2=\frac{1}{2}\|\Theta'-\Theta\|$.
  \end{enumerate*}
\end{center}
\end{lemma}
\begin{proof}
This is a slight modification of the statement of \cite[Lemma 5.3]{GPSS21}. 
The same proof works. 
\end{proof}


\begin{theorem}\label{thm:bounded_perturbation}
Let $d = 2n$ and let $\Theta$ be a real nonsingular antisymmetric $d \times d$ matrix.\footnote{\label{fn:ev-in-pairs}Note that a real antisymmetric $d\times d$-matrix $\Theta$ with odd $d$ can never be nonsingular; indeed, since $i\Theta$ is hermitian, the nonzero eigenvalues of $\Theta$ must be purely imaginary, and since $\Theta$ is real they must come in pairs $\lambda,\overline\lambda\in i\mathbb R$. 
Thus, if $d$ is odd, $0$ is necessarily an eigenvalue.}
Let $P_1, \ldots, P_d$ be the generators of one-parameter unitary groups $u_1, \ldots, u_d$ that commute according to $\Theta$. 
For any real antisymmetric $d \times d$ matrix $\Theta'$, the infinite ampliation $P^\infty:=P\otimes 1_{\ell^2(\bN)}$ of $P$ is a bounded perturbation of a $d$-tuple $Q$ of selfadjoint operators that generate $d$ unitary groups that commute according to $\Theta'$ such that
\[
\|P_k^{\infty} - Q_k\| \leq {\textstyle\frac{5}{\sqrt{2}}} \|\Theta - \Theta'\|^{1/2}
\]
for all $k=1, \ldots, d$. If $P$ is of infinite multiplicity, $P^\infty$ can be replaced by $P$.
\end{theorem}

\begin{proof}
Let $x_k,y_k,z_k$ denote the vectors from \Cref{lem:vectors-xyz}, and form the unitary groups 
\[
u_\Theta = (W(tx_1),\ldots,W(tx_d)) \quad \textrm{ and }  \quad u_{\Theta'} = (W(tz_1),\ldots,W(tz_d)), 
\]
which commute according to $\Theta$ and $\Theta'$, respectively (\Cref{lem:exitThetacom}). 
From
\[
p_{\Gamma(H)}W(tz_k)|_{\Gamma(H)}=e^{\frac{-\|p^\perp tz_k\|^2}{2}}W(p_Htz_k)=e^{-t^2\frac{\|\Theta-\Theta'\|}{4}}W(tx_k)
\]
we have $u_\Theta \prec_c u_{\Theta'}$ for $c = \frac{\|\Theta-\Theta'\|}{4}$. 

Now we apply \Cref{lem:vectors-xyz} with $z_1, \ldots, z_d$ in place of $x_1, \ldots, x_d$ and the roles of $\Theta$ and $\Theta'$ reversed, to find another $d$-tuple $v_\Theta = (v_1, \ldots, v_d)$ of unitary groups that commute according to $\Theta$ and satisfy $u_{\Theta'} \prec_c v_\Theta$ with $c = \frac{\|\Theta-\Theta'\|}{4}$. 
But by the Stone-von Neumann theorem\footnote{More precisely, its simple consequence that there is a unique irreducible representation of the algebra generated by $d$ unitary groups that commute according to a nonsingular $\Theta$, see \cite[Proposition 3.1]{Gao18} and, consequently, every infinite multiplicity representation is unitarily equivalent to the infinite ampliation thereof. The most common formulation of the Stone-von Neumann Theorem, see e.g.\ \cite[Theorem 14.8]{HallBook}, only refers to the special case $\Theta=
  \begin{pmatrix}
    0& -I\\
    I&0
  \end{pmatrix}$.}, $u_\Theta \sim v_\Theta$ and so $u_{\Theta'} \prec_c u_\Theta$. 
It follows that $d_{\mathrm D}(u,v) \leq \frac{\|\Theta-\Theta'\|}{4} < \infty$. 
By \Cref{thm:bdd_pert} we conclude that there exist strongly continuous $d$-tuples $U \sim u_\Theta$ and $V \sim u_{\Theta'}$ of infinite multiplicity whose generators are perturbations of one another by bounded operators of norm less than or equal to $\frac{5}{\sqrt{2}} \|\Theta - \Theta'\|^{1/2}$. 
By the Stone-von Neumann theorem again, the generators of $U$ are (unitarily equivalent to) an infinite ampliation of $P_1, \ldots, P_d$.

If $P$ is of infinite multiplicity, then $P\sim_{\mathrm{ue}}P^{\infty}$.
\end{proof}

In particular, the infinite ampliation of $P$ is a bounded perturbation of $d$ strongly commuting selfadjoint operators. 

Note that the above result was obtained by Gao \cite[Theorem 3.2]{Gao18} with the bound 
\[
\|P_k - Q_k\| \leq 9(d-1)\max_\ell |\theta_{k,\ell}-\theta'_{k,\ell}|^{1/2} 
\]
(here and for the rest of the discussion we assume that $P$ is of infinite multiplicity).
The bounds are not directly comparable because in Gao's result the bound depends on $k$. However, if $\max_\ell |\theta_{k,\ell}-\theta'_{k,\ell}|$ is independent of $k$ (for example if $\theta_{k,\ell}=\theta$ for all $k<\ell$), or if we only compare our result with
\[\max_k\|P_k-Q_k\|=9(d-1)\max_{k,\ell}|\theta_{k,\ell}-\theta'_{k,\ell}|^{1/2},\] our bound is significantly better for large $d$. Indeed, for the real antisymmetric matrix  $\Gamma:=\Theta-\Theta'$, whose nonzero eigenvalues always come in pairs $\lambda, \overline\lambda\in i\mathbb R$ (cf.\ Footnote \ref{fn:ev-in-pairs}),
\[\|\Theta-\Theta'\|^2=\|\Gamma\|^2=\max_{\lambda\in\sigma(\Gamma)} |\lambda|^2\leq \frac{1}{2} \sum_{\lambda\in\sigma(\Gamma)} |\lambda|^2 = \frac{\|\Gamma\|_{\mathrm{HS}}^2}{2}\leq \frac{d(d-1)}{2}\max_{k,\ell}  |\theta_{k,\ell}-\theta'_{k,\ell}|^2   \]
and we can thus achieve by \Cref{thm:bounded_perturbation}
\[\|P_k-Q_k\|\leq \frac{5}{\sqrt2} \sqrt[4]{\frac{d(d-1)}{2}}  \max_{k,\ell}  |\theta_{k,\ell}-\theta'_{k,\ell}|^{1/2}< 3 \sqrt{d}\max_{k,\ell}  |\theta_{k,\ell}-\theta'_{k,\ell}|^{1/2}\]
for all $k$. The difference is even more striking in the special case $\Theta=\theta J$ and 
$\Theta'=
\theta' J$ with 
$J=
\begin{pmatrix}
  0& - I\\ I&0
\end{pmatrix}$ and $\theta, \theta' \in\bR$, where we achieve
\[\|P_k-Q_k\|\leq \frac{5}{\sqrt{2}}|\theta-\theta'|^{1/2}\]
for all $k$, which does not depend on $d$ at all.

Letting $d = 2$ and $\theta_{12} = - \theta_{21} = 1$ we recover \cite[Theorem 3.1]{HR95} with the somewhat better bound $\frac{5}{\sqrt{2}}\approx 3.54$ (in the remark on page 637 of their paper, Haagerup and R{\o}rdam note that they can improve the constant $9$ that appears in \cite[Theorem 3.1]{HR95} to $\sqrt{45}\approx 6.71$).

It is also noteworthy that our estimate in  \Cref{thm:bounded_perturbation} does not depend on $d$. This makes some hope that with our methods a similar result can be obtained for $d=\infty$. However, the infinite-dimensional analogue of the Stone-von Neumann theorem is known to be false, see e.g.\ Summer's nice survey \cite{Sum01}. Therefore, the proof does not directly generalize, not even if one assumes the $u_k$ to be given as Weyl operators in Fock representation $u_k(t)=W(tx_k), x_k\in\ell^2(\bN),$ from the start.

\section*{Acknowledgements}
We thank the anonymous referee for several helpful comments and corrections which improved the quality of this manuscript.
MG thanks Franz Luef for discussions about the role of noncommutative tori in time-frequency analysis and possible applications of dilation techniques, which actually stimulated this work. MG is also grateful for the received hospitality during his visit to the Technion in May 2022.

\bibliographystyle{myplainurl}
\bibliography{GS}

\providecommand{\bysame}{\leavevmode\hbox to3em{\hrulefill}\thinspace}
\providecommand{\MR}{\relax\ifhmode\unskip\space\fi MR }
\providecommand{\MRhref}[2]{%
  \href{http://www.ams.org/mathscinet-getitem?mr=#1}{#2}
}
\providecommand{\href}[2]{#2}
\begin{thebibliography}{1}

\bibitem{ArvBook}
W.~Arveson, \emph{An invitation to {$C\sp*$}-algebras}, Graduate Texts in
  Mathematics, No. 39, Springer-Verlag, New York-Heidelberg, 1976, \href
  {https://doi.org/10.1007/978-1-4612-6371-5}
  {\path{doi:10.1007/978-1-4612-6371-5}}.

\bibitem{Gao18}
L.~Gao, \emph{Continuous perturbations of noncommutative {E}uclidean spaces and
  tori}, J. Operator Theory \textbf{79} (2018), no.~1, 173--200, \href
  {https://doi.org/10.7900/jot} {\path{doi:10.7900/jot}}.

\bibitem{GPSS21}
M.~Gerhold, S.~K. Pandey, O.~M. Shalit, and B.~Solel, \emph{Dilations of
  unitary tuples}, J. Lond. Math. Soc. (2) \textbf{104} (2021), no.~5,
  2053--2081, \href {https://doi.org/10.1112/jlms.12491}
  {\path{doi:10.1112/jlms.12491}}.

\bibitem{GS20}
M.~Gerhold and O.~M. Shalit, \emph{Dilations of {$q$}-commuting unitaries},
  Int. Math. Res. Not. IMRN \textbf{2022} (2020), no.~1, 63--88, \href
  {https://doi.org/10.1093/imrn/rnaa093} {\path{doi:10.1093/imrn/rnaa093}}.

\bibitem{HR95}
U.~Haagerup and M.~R{\o}rdam, \emph{Perturbations of the rotation
  {$C^\ast$}-algebras and of the {H}eisenberg commutation relation}, Duke Math.
  J. \textbf{77} (1995), no.~3, 627--656, \href
  {https://doi.org/10.1215/S0012-7094-95-07720-5}
  {\path{doi:10.1215/S0012-7094-95-07720-5}}.

\bibitem{HallBook}
B.~C. Hall, \emph{Quantum theory for mathematicians}, Graduate Texts in
  Mathematics, vol. 267, Springer, New York, 2013, \href
  {https://doi.org/10.1007/978-1-4614-7116-5}
  {\path{doi:10.1007/978-1-4614-7116-5}}.

\bibitem{Par12}
K.~R. Parthasarathy, \emph{An introduction to quantum stochastic calculus},
  Modern Birkh\"{a}user Classics, Birkh\"{a}user/Springer Basel AG, Basel,
  1992, [2012 reprint of the 1992 original], \href
  {https://doi.org/10.1007/978-3-0348-0566-7}
  {\path{doi:10.1007/978-3-0348-0566-7}}.

\bibitem{PetzBook}
D.~Petz, \emph{An invitation to the algebra of canonical commutation
  relations}, Leuven Notes in Mathematical and Theoretical Physics. Series A:
  Mathematical Physics, vol.~2, Leuven University Press, Leuven, 1990,
  available from \url{https://inspirehep.net/literature/308600}.

\bibitem{Sum01}
S.~J. Summers, \emph{On the {S}tone-von {N}eumann uniqueness theorem and its
  ramifications}, John von {N}eumann and the foundations of quantum physics
  ({B}udapest, 1999) (M.~R{\'e}dei and M.~St{\"o}ltzner, eds.), Vienna Circ.
  Inst. Yearb., vol.~8, Kluwer Acad. Publ., Dordrecht, 2001, pp.~135--152,
  \href {https://doi.org/10.1007/978-94-017-2012-0_9}
  {\path{doi:10.1007/978-94-017-2012-0_9}}.

\end{thebibliography}


\linespread{1}
\setlength{\parindent}{0pt}

\end{document}